\documentclass[reqno, oneside]{amsart}

\usepackage{chemarr}
\usepackage{amssymb}
\usepackage{hyperref}
\usepackage{comment}
\usepackage{tikz}
\usetikzlibrary{matrix,arrows,decorations}
\usepackage{tikz-cd}
\usepackage{cleveref}
\usepackage{adjustbox}
\usepackage[a4paper]{geometry}
\usepackage{dsfont}

\usepackage[textsize=scriptsize]{todonotes}

\theoremstyle{definition}
\newtheorem{nul}{}[section]
\newtheorem{dfn}[nul]{Definition}

\newtheorem{rmk}[nul]{Remark}

\newtheorem{cnstr}[nul]{Construction}
\newtheorem{ntn}[nul]{Notation}
\newtheorem{exm}[nul]{Example}

\newtheorem*{dfn*}{Definition}
\newtheorem*{axm*}{Axiom}
\newtheorem*{ntn*}{Notation}
\newtheorem*{exm*}{Example}
\newtheorem*{exr*}{Exercise}
\newtheorem*{int*}{Intuition}
\newtheorem*{qst*}{Question}
\newtheorem*{rmk*}{Remark}

\theoremstyle{plain}

\newtheorem{thm}[nul]{Theorem}
\newtheorem{prop}[nul]{Proposition}
\newtheorem{lem}[nul]{Lemma}
\newtheorem{var}[nul]{Variant}

\newtheorem{cor}{Corollary}[nul]
\newtheorem*{thm*}{Theorem}
\newtheorem*{prop*}{Proposition}
\newtheorem*{cor*}{Corollary}
\newtheorem*{lem*}{Lemma}
\newtheorem*{cnj*}{Conjecture}

\DeclareFontFamily{U}{rcjhbltx}{}
\DeclareFontShape{U}{rcjhbltx}{m}{n}{<->rcjhbltx}{}
\DeclareSymbolFont{hebrewletters}{U}{rcjhbltx}{m}{n}
\DeclareMathSymbol{\tsadi}{\mathord}{hebrewletters}{118}

\definecolor{DefColor}{rgb}{0.6,0.15,0.25}
\newcommand{\mdef}[1]{\textcolor{DefColor}{#1}}

\DeclareMathOperator{\Nm}{\mathrm{Nm}}

\DeclareMathOperator*{\colim}{\mathrm{colim}}
\DeclareMathOperator{\Map}{\mathrm{Map}}

\DeclareMathOperator{\fib}{\mathrm{fib}}
\DeclareMathOperator{\Hom}{\text{Hom}}

\newcommand{\one}{\mathds{1}}%

\def\bC{\mathbb{C}}

\def\Q{\mathbb{Q}}

\def\bS{\mathbb{S}}

\def\Z{\mathbb{Z}}
\def\C{\mathcal{C}}

\def\D{\mathcal{D}}

\DeclareMathOperator{\Fun}{\text{Fun}}

\DeclareMathOperator{\Mod}{\mathrm{Mod}}
\DeclareMathOperator{\Sp}{\mathrm{Sp}}
\newcommand{\Spaces}{\mathcal{S}}

\DeclareMathOperator{\Span}{\mathrm{Span}}
\DeclareMathOperator{\Fin}{\mathrm{Fin}}

\newcommand{\fin}{\mathrm{fin}}
\newcommand{\Vect}{\mathrm{Vect}}
\newcommand{\Cat}{\mathrm{Cat}}
\newcommand{\CAlg}{\mathrm{CAlg}}
\newcommand{\KU}{\mathrm{KU}}
\def\gl1{\mathrm{gl}_1}
\newcommand{\id}{\mathrm{id}}

\newcommand{\gp}{\mathrm{gp}}
\newcommand{\PrL}{\mathrm{Pr}^{\mathrm{L}}}
\newcommand{\Sfin}[2]{\Spaces_{#1}^{(#2)}}
\newcommand{\CMon}[2]{\mathrm{CMon}_{#1}^{(#2)}}
\newcommand{\PMon}[2]{\mathrm{PMon}_{#1}^{(#2)}}
\def\Flbar{\overline{\mathbb{F}}_{\ell}}
\def\Fl{\mathbb{F}_{\ell}}

\begin{document}

\begin{abstract}
We construct a lift of the $p$-complete sphere to the universal height $1$ higher semiadditive stable $\infty$-category $\tsadi_1$ of Carmeli--Schlank--Yanovski, providing a counterexample, at height $1$, to their conjecture that the natural functor $\tsadi_n \to \Sp_{T(n)}$ is an equivalence.  We then record some consequences of the construction, including an observation of T. Schlank that this gives a conceptual proof of a classical theorem of Lee on the stable cohomotopy of Eilenberg--MacLane spaces.  
\end{abstract}

\title{The sphere of semiadditive height 1}
\author{Allen Yuan}
\maketitle


\section{Introduction} 

Let $A$ be an abelian group with an action of a finite group $G$.  Then there is a canonical map
\[
\mathrm{Nm}_G: A_G \to A^G \]
given by the formula $[a] \mapsto \sum_{g\in G} ga$.  While  $\mathrm{Nm}_G$ is not an isomorphism in general, it is an isomorphism when $|G|$ acts invertibly on $A$, for instance when $A$ is a vector space over the rational numbers.  This phenomenon underlies Maschke's theorem on the semisimplicity of the representation theory of $G$ in characteristic zero.

The map $\mathrm{Nm}_G$ being an isomorphism for $\Q$-vector spaces is among the simplest examples of what has been dubbed \emph{ambidexterity}: certain colimits ($G$-orbits) and limits ($G$-fixed points) canonically agreeing.  In chromatic homotopy theory, one studies a whole family of generalizations of this example, known as the telescopic localizations\footnote{These are closely related to the localizations $\Sp_{K(n)}$ of spectra at the Morava $K$-theories.} $\Sp_{T(n)}$ of the $\infty$-category of spectra for each prime $p$ and integer $n\geq 0$.\footnote{The prime $p$ will be fixed and implicit throughout.}  
Building on work of  Hovey-Sadofsky \cite{HS96} and Greenlees-Sadofsky \cite{GS96}, Kuhn showed that:

\begin{thm}[Kuhn \cite{Kuhn}]\label{thm:kuhn}
Let $X$ be a $T(n)$-local spectrum with the action of a finite group $G$.  Then the canonical map $$\Nm_G: L_{T(n)}X_{hG} \to X^{hG}$$ is an equivalence.  
\end{thm}

Hopkins and Lurie interpreted this result as asserting that the $\infty$-category $\Sp_{T(n)}$ satisfies a higher categorical analogue of semiadditivity \cite{HL}.  To illustrate this, note that in a semiadditive category $\C$, given a finite set $T$ and a functor $F: T \to \C$, the colimit of $F$ (coproduct) and the limit of $F$ (product) are canonically equivalent.  Analogously, \Cref{thm:kuhn} asserts that for any map $X: BG \to \Sp_{T(n)}$ from a finite groupoid  $BG$ (i.e. an object with $G$-action), its colimit (homotopy orbits) and limit (homotopy fixed points) are canonically equivalent; in the language of \cite{HL}, $\Sp_{T(n)}$ is \emph{1-semiadditive}.  More generally, for $m\geq 0$, one can define a notion of an $m$-semiadditive $\infty$-category: roughly, colimits and limits over any finite $m$-type are required to agree.  


\begin{thm}[Carmeli--Schlank--Yanovski \cite{CSYTeleAmbi}]
The $\infty$-category $\Sp_{T(n)}$  is $\infty$-semiadditive; i.e., it is $m$-semiadditive for every $m\geq 0$.  
\end{thm}

This result implies the analogous result for $\Sp_{K(n)}$, which was first proved by Hopkins--Lurie \cite{HL}.   In fact, Carmeli--Schlank--Yanovski show that if $R$ is a nonzero $p$-local homotopy ring such that the corresponding localization $\Sp_R$ of spectra is $1$-semiadditive, then $\Sp_{K(n)} \subset \Sp_R \subset \Sp_{T(n)}$ for some $n$ \cite[Theorem B]{CSYTeleAmbi}.

Given these results, one can wonder whether these are essentially all examples of $p$-local higher semiadditive stable $\infty$-categories.  This question was studied in detail in \cite{CSYHeight}, where Carmeli--Schlank--Yanovski formulate a (purely categorical) notion of an $\infty$-semiadditive $\infty$-category being ``of semiadditive height $n$.''  One of the defining features of such an $\infty$-category is that the cohomology of any $n$-connected $\pi$-finite space\footnote{A space $X$ is $\pi$-finite if it has finitely many components and on each of them, $\pi_*(X)$ is a finite group.} vanishes (cf. \Cref{rmk:trunc}).  They show that $\Sp_{T(n)}$ is of semiadditive height $n$ and moreover that there is a ``universal'' such $\infty$-category:

\begin{thm}[{\cite[Theorem F]{CSYHeight}}]
For $n\geq 0$, there exists an idempotent commutative algebra $\tsadi_n$\footnote{The letter $\tsadi$, pronounced ``tsadi,'' is the first letter in the Hebrew word for ``color.''} in the $\infty$-category of presentable $\infty$-categories such that a presentable $\infty$-category $\C$ is stable, $p$-local, and $\infty$-semiadditive of height $n$ if and only if $\C$ admits a (necessarily unique) structure of a module over $\tsadi_n$.  
\end{thm}

We will give a more contextualized introduction to this theorem in \Cref{sec:prelim}. 
For now, we note that $\tsadi_n$ is in particular initial among presentable stable $p$-local symmetric monoidal $\infty$-categories which are $\infty$-semiadditive of height $n$, and so it fits into a diagram of adjoint pairs
\[
\begin{tikzcd}
 & & \tsadi_n\arrow[rrd, bend left=25, "L_{T(n)}^{\tsadi}"] \arrow[lld, "U_{\tsadi_n}"] & & \\
 \Sp \arrow[rrrr,"L_{T(n)}"'] \arrow[rru, bend left=25,"L_{\tsadi_n}"] & & & &\Sp_{T(n)}, \arrow[llu,"U_{T(n)}^{\tsadi}"]\arrow[llll, bend left=20,"U_{T(n)}"]
\end{tikzcd}
\]
with colimit preserving symmetric monoidal left adjoints displayed on top.  With reference to this diagram, \cite{CSYHeight} showed that $\tsadi_n$ shares some striking similarities with $\Sp_{T(n)}$:
\begin{enumerate}
\item For $n\geq 1$, the unique symmetric monoidal colimit preserving functor $L_{\tsadi_n}: \Sp \to \tsadi_n$ vanishes on bounded above spectra \cite[Proposition 5.3.9]{CSYHeight}.   
\item For $n\geq 1$, the right adjoint of the unique symmetric monoidal colimit preserving functor $\Spaces \to \tsadi_n$ from the $\infty$-category of spaces is conservative \cite[Corollary 5.3.10]{CSYHeight}.  
\end{enumerate}
For instance, (2) holds with $\tsadi_n$ replaced by $\Sp_{T(n)}$ because of the existence of the Bousfield--Kuhn functor $\Phi_n: \Spaces \to \Sp_{T(n)}$.  Moreover:
\begin{enumerate}
\item[(3)] $L_{T(0)}^\tsadi$ is an equivalence of categories \cite[Example 5.3.7]{CSYHeight}.
\item[(4)] $L_{T(n)}^\tsadi$ is a smashing localization for any $n\geq 0$ \cite[Corollary 5.5.14]{CSYHeight}.
\end{enumerate}
Given this evidence, Carmeli--Schlank--Yanovski conjectured that the functor $L_{T(n)}^{\tsadi}$ might be an equivalence for $n>0$ as well.  

The main result of this note is a counterexample to this conjecture in the case $n=1$.  Note that for $L_{T(n)}^{\tsadi}$ to be an equivalence, $U_{\tsadi_n}$ must be identified with $U_{T(n)}$, the inclusion of $T(n)$-local spectra.   Our main result is:

%


\begin{thm} \label{thm:main}
Let $\bS_p$ denote the $p$-completion of the sphere spectrum.  Then there exists a commutative algebra $\bS^{\tsadi_1}_p \in \CAlg(\tsadi_1)$ together with an equivalence $U_{\tsadi_1}(\bS^{\tsadi_1}_p) \simeq \bS_p$ of commutative algebras in spectra.  
\end{thm}

Our proof of this theorem, which is an application of the Segal conjecture (now a theorem of Carlsson \cite{Carlsson}), is given at the beginning of \Cref{sec:cnstr}.
Since $\bS_p$ is not $T(1)$-local\footnote{By the truth of the telescope conjecture at height $1$, this is the same notion as $K(1)$-local.}, \Cref{thm:main} implies that $L_{T(1)}^{\tsadi}$ cannot be an equivalence.  Nevertheless, since $L_{T(1)}^\tsadi$ is a smashing localization, $\Sp_{T(1)}$ sits as a full subcategory of $\tsadi_1$; in fact, we show in \Cref{sub:T1spherealg} that away from the prime $2$, the $T(1)$-local sphere $\bS_{T(1)}$ is an algebra over $\bS^{\tsadi_1}_p$.

\subsection{Applications}\label{sub:app}

The fact that $\bS_p$ lifts to an $\infty$-semiadditive category of semiadditive height $1$ turns out to have a number of interesting consequences which the author learned from S. Carmeli, T. Schlank, and L. Yanovski and which he thanks them for encouraging him to share. For instance, in \Cref{sub:cycl}, we discuss a consequence of \Cref{thm:main} for the higher cyclotomic extensions of \cite{CSYCyc}.  For the purposes of this introduction, we highlight a surprisingly immediate corollary of \Cref{thm:main} which a priori has nothing to do with ambidexterity: T. Schlank has observed that it gives a conceptual proof of the following classical theorem of Lee about the stable cohomotopy of Eilenberg--MacLane spaces.


\begin{cor}[{Lee \cite{Lee}}]\label{thm:lee}
Let $A$ be a simply connected $\pi$-finite space.  Then the natural map $A \to *$ induces an equivalence $\bS \simeq \bS^A$
on Spanier-Whitehead duals.   
\end{cor}
\begin{proof}
The statement is true with $\bS$ replaced by any rational ring spectrum, because $A$ is finite type and rationally trivial.  Therefore, by Sullivan's arithmetic fracture square, it suffices to prove the statement with $\bS$ replaced by its $p$-completion, where we may assume without loss of generality that $A$ has $p$-power torsion homotopy groups.  Then, since $U_{\tsadi_1}$ preserves limits, \Cref{thm:main} gives an equivalence
\[
\bS_p^A \simeq U_{\tsadi_1}(\bS^{\tsadi_1}_p)^A \simeq U_{\tsadi_1}((\bS^{\tsadi_1}_p)^A).
\]
But the $\infty$-category $\tsadi_1$ is $p$-typically $1$-semiadditive \emph{of semiadditive height $1$}, so the map $A\to *$ induces an equivalence $(\bS^{\tsadi_1}_p)^A \simeq \bS^{\tsadi_1}_p$ in $\tsadi_1$ (cf. \Cref{rmk:trunc}).  We therefore conclude, as this implies
\[
U_{\tsadi_1}((\bS^{\tsadi_1}_p)^A) \simeq U_{\tsadi_1}(\bS^{\tsadi_1}_p) \simeq \bS_p.
\]
\end{proof}

\begin{rmk} 
As was pointed out to the author by S. Carmeli, a similar proof shows that the spectrum of ``strict $p^r$-th roots of unity'' of $\bS_p$
\[
\mu_{p^r}(\bS_p) := \Hom_{\Sp}(\Z/p^r  , \gl1 \bS_p) = \Hom_{\CAlg(\Sp)}(\bS[\Z/p^r], \bS_p)
\]
has vanishing $\pi_k$ for $k\geq 2$; namely, because $\bS_p \simeq U(\bS^{\tsadi_1}_p)$, we have
\begin{align*}
\pi_2\Hom_{\CAlg(\Sp)}(\bS[\Z/p^r], \bS_p) &\cong \pi_0 \Hom_{\CAlg(\Sp)}(\bS[B^2\Z/p^r], \bS_p) \\
&\cong \pi_0 \Hom_{\tsadi_1}(L_{\tsadi_1}\bS[B^2 \Z/p^r], \bS^{\tsadi_1}_p) \cong 0,
\end{align*}
where the last isomorphism uses that $B^2 \Z/p^r$ is simply connected and $\pi$-finite (cf. \Cref{rmk:trunc}).   In fact, the existence of $\bS_p^{\tsadi_1}$ has even stronger consequences for the study of the strict units of $\bS$ and related objects, which will appear in future joint work with S. Carmeli and T. Nikolaus.  
\end{rmk}

\subsection{Acknowledgments}
The author would like to thank Shachar Carmeli, Tomer Schlank, and Lior Yanovski for inspiring discussions related to this material, and especially for teaching the author about the applications of \Cref{thm:main}.  He would also like to thank Clark Barwick and Jacob Lurie for helpful conversations and Shachar Carmeli and Thomas Nikolaus for comments on a draft.  The author was supported in part by NSF grant DMS-2002029.

\section{Universal higher semiadditive categories}\label{sec:prelim}

For the reader's convenience, we start with a brief review of the relevant definitions (\S\ref{sub:prelim1}, \S\ref{sub:prelim2}), culminating in the construction of $\tsadi_n$, a certain universal stable semiadditive category of height $n$.  The reader is encouraged to consult \cite{Harpaz} and \cite{CSYHeight} for a more systematic treatment of these ideas. 

The motivation behind the main construction is given in \S\ref{sub:ht1}; this section is intended to be as concrete as possible, and the author encourages the reader to start there, referring to previous sections as needed.  The technical aspects of the main construction, including compatibility with multiplicative structure, are handled in \S\ref{sub:technical}.

\subsection{Higher commutative monoids}\label{sub:prelim1}
Recall that if $X$ is an object in a semiadditive category, then $X$ acquires the (unique) structure of a commutative monoid: for any finite set $T$, one has a canonical $T$-fold addition map
\[
\int_T: X^T = \prod_T X \cong \coprod_T X \to X
\]
where the middle isomorphism comes from the fact that products and coproducts agree in a semiadditive category.  One way to phrase the definition of a commutative monoid is:

\begin{dfn}
Let $\Span(\Fin)$ denote the $(2,1)$-category of finite sets and spans, and let $\C$ be an $\infty$-category with products.  Then a \mdef{commutative monoid} in $\C$ is a functor 
\[
M: \Span(\Fin) \to \C
\]
satisfying the \emph{Segal condition}: i.e., for any $T\in \Span(\Fin)$, the collection  $\{ \rho_t = (T \xleftarrow{t} * \rightarrow *)\}_{t\in T}$ of arrows induces an equivalence
\[
\rho : M(T) \simeq M^T.
\]
\end{dfn}

In more concrete terms, $M(*)$ is the underlying object in $\C$ of the commutative monoid, and the functoriality in spans encodes the commutative monoid structure.  This definition has the feature that it generalizes easily to capture the algebraic structure enjoyed by objects in \emph{higher} semiadditive $\infty$-categories.  Namely, if $X$ is now an object in an $m$-semiadditive $\infty$-category, then not only does $X$ admit the structure of a commutative monoid, but for any $m$-truncated\footnote{That is, $\pi_*=0$ for $*>m$.} $\pi$-finite space $A$, one has an ``$A$-fold'' addition map
\[
\int_A: X^A = \lim_A X \cong \colim_A X \to X
\]
where the middle equivalence comes from the fact that colimits and limits over $m$-truncated $\pi$-finite spaces agree by $m$-semiadditivity.  These ``higher'' addition maps can be organized into an enhancement of a commutative monoid structure known as an \emph{$m$-commutative monoid}.  For our purposes, it will be convenient to work with a $p$-typical version of the story, so we state the definitions in this context:

\begin{ntn}
For $m\geq 0$, let $\mdef{\Sfin{m}{p}}$ denote the $\infty$-category of $m$-truncated $p$-finite spaces; i.e., spaces $A$ such that $\pi_0(A)$ is finite and the higher homotopy groups are finite $p$-groups concentrated in degrees $[1,m]$.  Let $\mdef{\Span(\Sfin{m}{p})}$ denote the $\infty$-category of spans of such spaces (cf. \cite[\S 5]{Barwick}). 
\end{ntn}

\begin{dfn}\label{dfn:cmonm}
A \mdef{$p$-typical pre-$m$-commutative monoid} in a presentable $\infty$-category $\C$ is a functor
\[
M: \Span(\Sfin{m}{p}) \to \C.
\]
We say that $M$ is a \mdef{$p$-typical $m$-commutative monoid} if it additionally satisfies the \emph{Segal condition}, that for any $m$-truncated $p$-finite space $A$, the collection of arrows $\{\rho_a = (A \xleftarrow{a} * \rightarrow * )\}_{a\in A}$ induces an equivalence
\[
\rho : M(A) \simeq M(*)^A.
\]
We denote by \mdef{$\CMon{m}{p}(\C) \subset \PMon{m}{p}(\C)$} the $\infty$-categories of $p$-typical $m$-commutative monoids and pre-$m$-commutative monoids in $\C$.
\end{dfn}

There is a forgetful functor $\CMon{m}{p}(\C) \to \C$ by $M \mapsto M(*)$ and we think of the functor $M(-)$ as equipping $M(*)$ with extra structure.  A particularly accessible part of this structure is:

\begin{cnstr}\label{cnstr:cardinality}
Suppose $M \in \CMon{m}{p}(\C)$.  Then, given any $m$-truncated $p$-finite space $A$, the canonical morphism $(* \leftarrow A \rightarrow *)$ in $\Span(\Sfin{m}{p})$ determines a natural map
\begin{equation*}\label{eqn:card}
\mdef{|A|_M:  M(*)\to M(*)}.
\end{equation*}

This construction determines a natural endomorphism 
\[
|A|:\id_{\CMon{m}{p}(C)} \to \id_{\CMon{m}{p}(C)}
\]
of the identity functor of $\CMon{m}{p}(C)$, which we refer to as the \emph{cardinality of $A$}, because in the special case when $A$ is a finite set, it is given by multiplication by the cardinality of that set.  
\end{cnstr}


\subsection{Higher semiadditivity and height}\label{sub:prelim2}

 It turns out that the theory of $m$-commutative monoids is intimately related to $m$-semiadditivity.  To state the relationship cleanly, we need the following notion:

\begin{dfn}[{\cite[Definition 4.8.2.1]{HA}}]\label{dfn:idem}
Given a symmetric monoidal $\infty$-category $(\C, \one)$, we say a morphism $u: \one \to A\in \C$ exhibits $A$ as an idempotent object if $\id_A\otimes u: A \to A\otimes A$ is an equivalence.  In this case, $A$ admits a unique commutative algebra structure with unit $u$ and we call the resulting $A\in \CAlg(\C)$ an \mdef{idempotent algebra}.  
\end{dfn}

The central feature of idempotent algebras $A$ is that the forgetful functor $\Mod_A(\C)\to \C$ is fully faithful; that is, admitting the structure of an $A$-module is a property of an object of $\C$, and such an $A$-module structure is necessarily unique \cite[Proposition 4.8.2.10]{HA}.


We will be interested in idempotent algebras in the symmetric monoidal $\infty$-category $\PrL$ of presentable $\infty$-categories with the Lurie tensor product \cite[Proposition 4.8.1.15]{HA}. 

\begin{rmk}\label{rmk:idem}

Idempotent algebras $\mathcal{M}\in \CAlg(\PrL)$ are known as \emph{modes} and have been studied in \cite[\S 4.8.2]{HA}, \cite[\S 5]{CSYHeight}.  As explained in \Cref{dfn:idem}, admitting the structure of an $\mathcal{M}$-module is a property of an $\infty$-category $\C\in \PrL$ -- one says that $\mathcal{M}$ is a \emph{mode} classifying that property.  
For instance, the $\infty$-category $\Sp$ of spectra is idempotent and a presentable $\infty$-category is a module over $\Sp$ if and only if it is stable, so $\Sp$ is the mode classifying stability.  
\end{rmk}

\begin{prop}\label{prop:cmon_props}\hfill

\begin{enumerate}
\item The $\infty$-category $\CMon{m}{p}(\Spaces)$ is an idempotent algebra in $\PrL$.
\item An $\infty$-category $\C\in\PrL$  admits the structure of a module over $\CMon{m}{p}(\Spaces)$ if and only if $\C$ is $p$-typically $m$-semiadditive.
\item For $\C\in \PrL$, there is a natural equivalence of $\infty$-categories 
\[
\CMon{m}{p}(\Spaces) \otimes \C \simeq \CMon{m}{p}(\C).
\]
\item Consequently, an $\infty$-category $\C\in\PrL$ is $p$-typically $m$-semiadditive if and only if the forgetful functor $\CMon{m}{p}(\C) \to \C$ is an equivalence. 
\end{enumerate}
\end{prop}

\begin{proof}
In the non-$p$-typical case, (1) and (2) are discussed in \cite[\S 5.2]{Harpaz}, (3) is \cite[Proposition 5.3.1]{CSYHeight}, and (4) is \cite[Corollary 5.15]{Harpaz}; the proof in the $p$-typical case is identical.  
\end{proof}

Thus, in the language of \Cref{rmk:idem}, $\CMon{m}{p}(\Spaces)$ is the mode classifying $p$-typical $m$-semiadditivity.  A consequence of \Cref{prop:cmon_props}(4) is that if $\C$ is $p$-typically $m$-semiadditive, then the cardinality construction of \Cref{cnstr:cardinality} gives natural endomorphisms 
\[
|A|: \id_{\C} \to \id_{\C}
\]
for any $m$-truncated $p$-finite space $A$.  These invariants of $\C$ can be used to define the \emph{semiadditive height} of $\C$.

\begin{dfn}[{\cite[Definition 3.1.11]{CSYHeight}}]
For an $\infty$-category $\C$ and $\alpha : \id_{\C} \to \id_{\C}$ a natural endomorphism of the identity functor, we say that:
\begin{enumerate}
\item $X\in \C$ is \mdef{$\alpha$-divisible} if $\alpha_X$ is an equivalence.
\item $X\in \C$ is \mdef{$\alpha$-complete} if $\Map(Z,X) = 0 $ for all $\alpha$-divisible $Z$.  
\end{enumerate}
\end{dfn}
\begin{dfn}
We say that a $p$-typical $m$-semiadditive $\infty$-category $\C$ has \mdef{semiadditive height $n\leq m$} if
\begin{enumerate}
\item Any $X\in \C$ is $|B^{n-1}C_p|$-complete.\footnote{This automatically implies that $X$ is $|B^kC_p|$-complete for $k\leq n-1$, and in particular $p$-complete \cite[Proposition 3.1.9]{CSYHeight}.}
\item Any $X\in \C$ is $|B^n C_p|$-divisible.  
\end{enumerate}
\end{dfn}

\begin{rmk}[{\cite[Proposition 3.2.3, Remark 2.4.6]{CSYHeight}}]\label{rmk:trunc}
Intuitively, a $p$-typical $m$-semiadditive $\infty$-category $\C$ of semiadditive height $n$ ``only sees up to the $n$th homotopy group'' in the sense that for any $X\in \C$ and $n$-connected $p$-finite space $A$, the unique map $A\to *$ induces an equivalence
\[
X \to X^A.
\]
\end{rmk}

With these definitions, the stable $\infty$-semiadditive $\infty$-categories $\Sp_{K(n)}$ and $\Sp_{T(n)}$ have semiadditive height $n$.  This paper is concerned with the universal such $\infty$-category:

\begin{thm}[{\cite[Theorem F]{CSYHeight}}]
For $n\geq 0$, there exists an idempotent algebra $\tsadi_n \in \CAlg(\PrL)$ such that an $\infty$-category $\C \in \PrL$ admits the structure of a module over $\tsadi_n$ if and only if $\C$ is stable, $p$-local, and $\infty$-semiadditive of height $n$.  
\end{thm}

The $\infty$-category $\tsadi_n$ admits the following concrete construction.  

\begin{prop}\label{thm:tsadi_desc}
For $n\geq 0$, $\tsadi_n$ can be identified with the full subcategory
\[
\D \subset \CMon{n}{p}(\Sp_{(p)})
\]
of $p$-typical $n$-commutative monoids in $p$-local spectra which are $|B^{n-1}C_p|$-complete and $|B^nC_p|$-divisible.  
\end{prop}
\begin{proof}
Note that by the proof of \cite[Theorem 5.3.6]{CSYHeight}, $\D$ is the mode classifying the property of being stable, $p$-local, and \emph{$p$-typically} $\infty$-semiadditive of height $n$.   But by \cite[Theorem 3.2.6]{CSYHeight}, $p$-typical $n$-semiadditivity coincides with $n$-semiadditivity for any presentable $p$-local $0$-semiadditive category, so this coincides with the property classified by $\tsadi_n$.  
\end{proof}

\begin{exm}
In the case $n=0$, we have that $\CMon{0}{p}(\Sp_{(p)}) = \mathrm{CMon}(\Sp_{(p)}) = \Sp_{(p)}$, and $\tsadi_0 \subset \Sp_{(p)}$ is the full subcategory on which $|B^0C_p| = p$ acts invertibly, i.e., $\tsadi_0 = \Sp_{\Q}$.  
\end{exm}

\subsection{The case of height 1}\label{sub:ht1}\hfill

To motivate our main construction, we now reflect on the case of height $1$ (where the construction takes place).


\begin{exm}\label{exm:tsadi_obj}
In the case $n=1$, \Cref{thm:tsadi_desc} implies that the data of an object $X\in \tsadi_1$ is a functor $X:  \Span(\Sfin{1}{p}) \to \Sp_{(p)}$ with the following properties:
\begin{enumerate}
\item $X(*)$ is $p$-complete (as  $|B^0C_p| = p$).  
\item The map $|BC_p| : X(*) \to X(*)$ induced by the span $(*\leftarrow BC_p \rightarrow *)$ is an equivalence.
\item The functor $X$ satisfies the Segal condition (cf. \Cref{dfn:cmonm}).  
\end{enumerate}
Via this description, the ``underlying'' functor $U_{\tsadi_1} : \tsadi_1 \to \Sp$ takes such a functor $X$ to the spectrum $X(*)$.  
\end{exm}

Having unwound the definition in this way, we see that objects of $\tsadi_1$ are not altogether unfamiliar.  

\begin{rmk}
The objects of $\Sfin{1}{p}$ are simply groupoids with $\pi_0$ finite and $\pi_1$ a finite $p$-group.  Thus, functors $\Span(\Sfin{1}{p}) \to \Sp$ are closely related to global equivariant spectra (in the sense of Schwede \cite{Schwede}) with the additional structure of \emph{deflations} or \emph{exotic transfers}, corresponding to the functoriality in spans such as $(BG \xleftarrow{\id} BG \to *)$.  
\end{rmk}

Motivated by this, we may describe the higher commutative monoid underlying $\KU_p$ in terms of global equivariant $K$-theory.

\begin{exm}\label{exm:KUp}
The $p$-completed complex $K$-theory spectrum $\KU_p$ is $T(1)$-local and therefore, by \Cref{prop:cmon_props}(4), it extends canonically and uniquely to a $p$-typical $1$-commutative monoid.
This $1$-commutative monoid has the following concrete description: letting $\Vect_{\bC}^{\fin}$ denote the (ordinary) category of finite dimensional complex vector spaces, one has a functor
\[
\Vect_{\bC}^{\fin}(-): \Span(\Sfin{1}{p}) \to \Cat_{\infty}
\]
defined on objects by $A\mapsto \Fun(A, \Vect_{\bC}^{\fin})$ and on a morphism $(B \xleftarrow{f} A \xrightarrow{g} C)$ by the functor
\[
g_{!}f^* : \Fun(B, \Vect_{\bC}^{\fin}) \to \Fun(C, \Vect_{\bC}^{\fin})
\]
of restriction along $f$ followed by left Kan extension along $g$.  Applying group completion and $K(1)$-localization point-wise, we obtain a functor 
\[
\KU_p(-) : \Span(\Sfin{1}{p}) \to \Sp.
\]
By Suslin's theorem \cite{Suslin1}, this functor satisfies $\KU_p(BG) \simeq \KU^G_p$, the $p$-complete $G$-equivariant $K$-theory spectrum.  By the Atiyah--Segal completion theorem, $\KU^G_p$ can be identified with $\KU_p^{BG}$ (as $G$ is a finite $p$-group); in other words, the functor $\KU_p(-)$ satisfies the Segal condition and therefore is the (necessarily unique) $1$-commutative monoid structure on $\KU_p$.  

This construction gives a computational handle on the $1$-commutative monoid structures on $K(1)$-local spectra.   For example, under $\Vect_{\bC}^{\fin}(-)$, the span $(*\leftarrow BG \rightarrow *)$ is sent to the composite
\[
\Vect_{\bC}^{\fin} \xrightarrow{\mathrm{triv}_G} \Fun(BG, \Vect_{\bC}^{\fin}) \xrightarrow{-/G} \Vect_{\bC}^{\fin}
\]
of giving a complex vector space the trivial $G$-action followed by quotienting by $G$. Since this composite is naturally the identity, we conclude that $|BG|$ is the identity on $\KU_p$.   In fact,  this technique can be extended to compute power operations and related phenomena in the $K(1)$-local sphere, cf. \cite{CY}.  

%
\end{exm}

The main construction behind \Cref{thm:main} is a variant of \Cref{exm:KUp} which replaces vector spaces with finite sets.  Hence, we dedicate the following section to generalizing this example.

\subsection{Higher commutative monoids from categories} \label{sub:technical}

\begin{ntn}
Let \mdef{$\Cat_{\Sfin{m}{p}}$} denote the $\infty$-category of small $\infty$-categories which admit $\Sfin{m}{p}$-shaped colimits and $\Sfin{m}{p}$-colimit preserving functors between them.  
\end{ntn}

Here, we give a procedure which associates to each $\C \in \Cat_{\Sfin{m}{p}}$ a certain functor 
\begin{align*}
\underline{K}(\C)(-) : \Span(\Sfin{m}{p}) \to \Sp.\\
A\mapsto ((\C^A)^{\simeq})^{\gp}
\end{align*}
In fact, the construction will be compatible with multiplicative structure, which we discuss in \S \ref{subsub:mult}.  We remark that the procedure here is a primitive version of the technology of \cite{ShayTomer} and \cite{CY}.

\begin{cnstr}\label{exm:catsemiadd}
The $\infty$-category $\Cat_{\Sfin{m}{p}}$ is $p$-typically $m$-semiadditive (\cite[Proposition 5.25]{Harpaz}, \cite[Proposition 2.2.7]{CSYHeight}).  Thus, by  \Cref{prop:cmon_props}(4), every $\C \in \Cat_{\Sfin{m}{p}}$ acquires a canonical and unique lift to a $p$-typical $m$-commutative monoid
\begin{align*}
\C^{(-)}: \Span(\Sfin{m}{p}) &\to \Cat_{\Sfin{m}{p}}, \\
A &\mapsto \C^A
\end{align*}
which we refer to as the \emph{coCartesian} higher commutative monoid structure.  On morphisms, $\C^{(-)}$ sends a span $(A\xleftarrow{f}B\xrightarrow{g} C)$ to the functor 
\[
g_!f^* : \C^A \to \C^C
\]
of restriction along $f$ followed by left Kan extension along $g$ \cite[\S 5]{ShayTomer}.  
\end{cnstr}

At this point, it remains to pass to $K$-theory (i.e., group completion) pointwise.  However, in preparation to analyze the multiplicative structure, we do this carefully in steps:

\begin{cnstr}\label{cnstr:k-functor}
Consider the functor $\underline{K}: \Cat_{\Sfin{m}{p}} \to \PMon{m}{p}(\Sp)$ defined by the composite
\begin{align*}
\Cat_{\Sfin{m}{p}} \xrightarrow{\simeq} \CMon{m}{p}(\Cat_{\Sfin{m}{p}}) \xrightarrow{\simeq} \CMon{m}{p}(\mathrm{CMon}(\Cat_{\Sfin{m}{p}})) &\xrightarrow{\subset} \PMon{m}{p}(\mathrm{CMon}(\Cat_{\Sfin{m}{p}}))\\
&\xrightarrow{(-)^{\simeq}} \PMon{m}{p}(\mathrm{CMon}(\Spaces))\\
&\xrightarrow{(-)^{\gp}} \PMon{m}{p}(\Sp)
\end{align*}
where
\begin{enumerate}
\item The first arrow is the coCartesian $p$-typical $m$-commutative monoid structure of \Cref{exm:catsemiadd}.
\item The second arrow is by applying \Cref{prop:cmon_props}, noting that there is an equivalence
\[
\CMon{m}{p}(\Spaces) = \CMon{m}{p}(\Spaces)\otimes \mathrm{CMon}(\Spaces)
\] 
because $\CMon{m}{p}(\Spaces)$ is semiadditive and therefore a module over the idempotent algebra $\mathrm{CMon}(\Spaces) = \CMon{0}{p}(\Spaces)$.
\item The third arrow is the canonical inclusion of $m$-commutative monoids into pre-$m$-commutative monoids.
\item The fourth arrow is by post-composition with the functor 
\[
\mathrm{CMon}(\Cat_{\Sfin{m}{p}}) \to \mathrm{CMon}(\Spaces)
\]
induced by taking maximal subgroupoid (noting that this preserves products).  
\item The final arrow is by post-composition with the group completion functor
\[
(-)^{\gp}:\mathrm{CMon}(\Spaces) \to \Sp.
\] 
 \end{enumerate}
\end{cnstr}

\begin{ntn}
Moving forward, we denote by \mdef{$K:\mathrm{CMon}(\Cat_{\Sfin{m}{p}}) \to \Sp$} the composite
\[
\mathrm{CMon}(\Cat_{\Sfin{m}{p}}) \xrightarrow{(-)^{\simeq}} \mathrm{CMon}(\Spaces) \xrightarrow{(-)^{\gp}} \Sp.
\]
\end{ntn}


\subsubsection{Multiplicative structure}\label{subsub:mult}

The Cartesian product of spaces gives $\Sfin{m}{p}$ a symmetric monoidal structure; this in turn induces a symmetric monoidal structure on $\Span(\Sfin{m}{p})$ which is given on objects by the formula $(A,B)\mapsto A\times B$.  Note that this is \emph{not} the Cartesian symmetric monoidal structure on $\Span(\Sfin{m}{p})$, which would be given by disjoint union at the level of spaces.  

\begin{cnstr}\label{cnstr:cmonsym}
For any presentably symmetric monoidal $\infty$-category $\C\in \CAlg(\PrL)$, the Day convolution endows $\PMon{m}{p}(\C) := \Fun(\Span(\Sfin{m}{p}),\C)$ with a symmetric monoidal structure (\cite{SaulDay}, \cite[Proposition 2.2.6.16]{HA}).   Moreover, the localization 
\[
L^{\mathrm{seg}} : \PMon{m}{p}(\C) \to \CMon{m}{p}(\C)
\]
which is left adjoint to the inclusion is compatible with this Day convolution monoidal structure \cite[Proposition 4.25]{ShayTomer}, and therefore $\CMon{m}{p}(\C)$ inherits a symmetric monoidal structure by \cite[Proposition 2.2.1.9]{HA}.  With these monoidal structures, $L^{\mathrm{seg}}$ is symmetric monoidal and the natural inclusion $\CMon{m}{p}(\C)\subset \PMon{m}{p}(\C) $ is lax symmetric monoidal \cite[Corollary 7.3.2.7]{HA}.  
\end{cnstr}

\begin{rmk}
There is, a priori, a different symmetric monoidal structure on $\CMon{m}{p}(\C)$ in the picture.  As $\CMon{m}{p}(\Spaces)$ is an idempotent algebra in $\PrL$ by \Cref{prop:cmon_props}, it admits a canonical commutative algebra structure and induces a commutative algebra structure on 
\[
\CMon{m}{p}(\C) = \CMon{m}{p}(\Spaces) \otimes \C
\]
for any $\C\in \CAlg(\PrL)$.  However, by work of Ben-Moshe--Schlank \cite[Theorem 4.27]{ShayTomer}, this symmetric monoidal structure coincides with the one of \Cref{cnstr:cmonsym}.  We may therefore refer to \emph{the} symmetric monoidal structure on $\CMon{m}{p}(\C)$, for $\C\in\CAlg(\PrL)$, without any ambiguity.
\end{rmk}

  We will be interested in the case $\C = \Cat_{\Sfin{m}{p}}$, which has a symmetric monoidal structure given by the Lurie tensor product (for the class of $\Sfin{m}{p}$-colimits)  \cite[Proposition 4.8.1.15]{HA}.  

\begin{prop}
The functor $\underline{K}: \Cat_{\Sfin{m}{p}} \to \PMon{m}{p}(\Sp)$ of \Cref{cnstr:k-functor} admits a canonical lax symmetric monoidal structure.  
\end{prop}
\begin{proof}
It suffices to check that each of the arrows of the composite of \Cref{cnstr:k-functor} admits a canonical lax symmetric monoidal structure.  The first two arrows are symmetric monoidal because they are induced by unit maps of idempotent algebras in $\PrL$.  The third arrow is lax symmetric monoidal by \Cref{cnstr:cmonsym}.  The last two arrows are given by post-composition, so by the properties of Day convolution \cite[Proposition 2.2.6.16]{HA}, it suffices to show that the functor $K$ given by the composite
\[
\mathrm{CMon}(\Cat_{\Sfin{m}{p}}) \xrightarrow{(-)^{\simeq}} \mathrm{CMon}(\Spaces) \xrightarrow{(-)^{\gp}} \Sp
\]
is lax symmetric monoidal.  The functor $(-)^{\gp}$ is symmetric monoidal, so it suffices to show it for $(-)^{\simeq}$; but $(-)^{\simeq}$ arises by applying $\mathrm{CMon}$ to the composite 
\[
\Cat_{\Sfin{m}{p}} \xrightarrow{\mathrm{incl}} \Cat_{\infty} \xrightarrow{(-)^{\simeq}} \Spaces
\]
so again by properties of Day convolution, it suffices to show that each of these are lax symmetric monoidal.  But the first is essentially by definition, and the second is because it has a symmetric monoidal left adjoint, given by inclusion.
\end{proof}

Therefore, the functor of \Cref{cnstr:k-functor} takes commutative algebras to commutative algebras.  Recalling that commutative algebras for the Day convolution are simply lax symmetric monoidal functors, we may summarize our work in this section as follows:

\begin{cor}\label{cor:full-k-functor}
Let $\C\in \CAlg(\Cat_{\Sfin{m}{p}})$.  Then we may functorially associate to $\C$ a lax symmetric monoidal functor
\[
\underline{K}(\C)(-) : \Span(\Sfin{m}{p}) \to \Sp
\]
given on objects by $\underline{K}(\C)(A) = ((\C^A)^{\simeq})^{\gp}$ and sending a morphism $(A\xleftarrow{f}B\xrightarrow{g} C)$ to the map induced by the functor $g_!f^*$ of restriction along $f$ followed by left Kan extension along $g$.
\end{cor}

\section{The $1$-commutative structure on $\bS_p$}\label{sec:cnstr}

We now make our main construction, which is a variant of \Cref{exm:KUp} with complex vector spaces replaced by finite sets.  By \Cref{exm:tsadi_obj}, \Cref{thm:main} will follow easily from the following more concrete statement: 

\begin{prop}\label{thm:mainprecise}
There exists a lax symmetric monoidal functor $\bS^{\tsadi_1}_p: \Span(\Sfin{1}{p}) \to \Sp_{(p)}$ such that:
\begin{enumerate}
\item There is an equivalence $\bS^{\tsadi_1}_p (*) = \bS_p$.  
\item The image of the span $(*\leftarrow BC_p \rightarrow *)$ under the functor $\bS^{\tsadi_1}_p$ is homotopic to the identity (in particular, an equivalence).
\item The functor $\bS^{\tsadi_1}_p$ satisfies the Segal condition (cf. \Cref{dfn:cmonm}).
\end{enumerate}
\end{prop} 

\begin{proof}
Let $\Fin$ denote the category of finite sets and note that $\Fin \in \Cat_{\Sfin{1}{p}}$.  In fact, as the product of finite sets commutes with colimits separately in each variable, we have $\Fin \in \CAlg(\Cat_{\Sfin{1}{p}})$.   Thus, \Cref{cor:full-k-functor} supplies a lax symmetric monoidal functor
\[
\underline{K}(\Fin): \Span(\Sfin{1}{p}) \to \Sp.
\]
Composing with $p$-completion, regarded as a lax symmetric monoidal functor $(-)^{\wedge}_p: \Sp \to \Sp_{(p)}$, we set $\bS^{\tsadi_1}_p = \underline{K}(\Fin)^{\wedge}_p$ and claim that it satisfies the conditions of the theorem:
\begin{itemize}
\item Condition (1) is just the fact that $\bS \simeq (\Fin^{\simeq})^{\gp}$ (i.e., the Barratt-Priddy-Quillen theorem).
\item For condition (2), we note that the span $(* \leftarrow BC_p \rightarrow*)$ is sent to the map induced on $p$-complete $K$-theory by the functor
\[
\Fin \xrightarrow{\mathrm{triv}_G} \Fun(BG, \Fin) \xrightarrow{-/G} \Fin
\]
where the first arrow gives a finite set the trivial $G$-action, and the second arrow quotients by the action of $G$.  This functor is naturally isomorphic to the identity functor, so the map induced on $K$-theory is also the identity. 
\item For condition (3), we must check that for any $A\in \Sfin{1}{p}$, the natural map
\[
K(\Fin^A)^{\wedge}_p \to (K(\Fin)^{\wedge}_p)^A
\]
is an equivalence.  Since both sides send disjoint unions in $A$ to products, it suffices to check in the special case that $A$ is connected.  Choosing an equivalence $A \simeq BG$ for a finite $p$-group $G$, we are left to consider the natural map
\[
K(\Fin^{BG})^{\wedge}_p \to (K(\Fin)^{\wedge}_p)^{BG}.
\]
But this is the map from the $G$-fixed points of the equivariant $p$-complete sphere to the $G$-homotopy fixed points, which is an equivalence by the Segal conjecture, as proved by Carlsson \cite{Carlsson} (cf. \cite{MayMcClure} for $p$-completion versus completion at the augmentation ideal).  
\end{itemize}
\end{proof}

\begin{proof}[Proof of \Cref{thm:main}]
By \Cref{exm:tsadi_obj} and the definition of the symmetric monoidal structure on $\CMon{m}{p}(\Sp_{(p)})$, $\bS^{\tsadi_1}_p$ defines a commutative algebra in $\tsadi_1$.  Since $U_{\tsadi_1}(\bS^{\tsadi_1}_p) = \bS_p$ has a unique commutative algebra structure, the theorem follows.  
\end{proof}

\subsection{$\bS^{\tsadi_1}_p$ and $T(1)$-local spectra}\label{sub:T1spherealg}

Recall that $\Sp_{T(1)}$ can be regarded as a full subcategory of $\tsadi_1$ via the right adjoint $U_{T(1)}^{\tsadi}$ of the localization.  We can make two straightforward variants of the construction of $\bS_p^{\tsadi_1}$ which live entirely in the $T(1)$-local category\footnote{We keep our justifications brief, as (more complicated versions of) these variants are discussed more carefully in \cite[\S 4]{CY}.}.  

\begin{var}\label{var:KUp}
Let $\ell\neq p$ be a prime and consider the functor
\[
L_{T(1)}\underline{K}(\Vect^{\fin}_{\Flbar}): \Span(\Sfin{1}{p}) \to \Sp_{T(1)}.
\]
By \cite[Theorem 0.3]{FHM} and the Atiyah-Segal completion theorem, this functor satisfies the Segal condition and takes the value $\KU_p$ at $*\in \Span(\Sfin{1}{p})$.  It is therefore the unique $p$-typical $1$-commutative monoid structure on $\KU_p$.
\end{var}

\begin{var}\label{var:sk1}
One can deduce from \cite{FHM} (cf. \cite[\S 4.1.2]{CY}) that via the identification of \Cref{var:KUp}, the map induced by Frobenius on $\Flbar$ coincides with the Adams operation $\psi^\ell: \KU_p \to \KU_p$, and that the functor
\[
L_{T(1)}\underline{K}(\Vect^{\fin}_{\Fl}): \Span(\Sfin{1}{p}) \to \Sp_{T(1)}
\]
exhibits the unique $p$-typical $1$-commutative monoid structure on the ring 
\[
\KU_p^{h\psi^\ell} = \fib(\KU_p \xrightarrow{\psi^\ell-1} \KU_p)
\]
of homotopy fixed points under the $\Z$-action determined by the Adams operation.  
\end{var}

By functoriality of our constructions, these $T(1)$-local spectra receive natural maps from $\bS^{\tsadi_1}_p$, and in particular we have:

\begin{prop}\label{prop:sk1}
Let $p$ be an odd prime and let $\bS_{T(1)}\in \CAlg(\Sp_{T(1)})$ denote the $T(1)$-local sphere.  Then $U_{T(1)}^{\tsadi}(\bS_{T(1)})$ admits the structure of a commutative algebra over $\bS^{\tsadi_1}_p$.  
\end{prop}
\begin{proof}
Choose a prime $\ell$ which is a generator of $\Z_p^{\times}$, so that $\KU_p^{h\psi^\ell} = \bS_{T(1)}$. 
Since the free functor $\Fin \to \Vect^{\fin}_{\Fl}$ is a map in $\CAlg(\Cat_{\Sfin{1}{p}})$, we obtain from \Cref{cor:full-k-functor} a natural transformation of functors 
\[
\underline{K}(\Fin)^{\wedge}_p \to \underline{K}(\Vect^{\fin}_{\Fl}) \to L_{T(1)}\underline{K}(\Vect^{\fin}_{\Fl}).
\]
Unwinding the definitions and applying \Cref{thm:mainprecise} and \Cref{var:sk1}, we obtain the desired statement.
\end{proof}

\begin{rmk}\label{rmk:sk1p=2}
When $p=2$, $\bS_{T(1)}$ does not arise as the fixed points of an Adams operation on $\KU_2$; instead, the proof of \Cref{prop:sk1} only shows that a certain $C_2$-Galois extension of $\bS_{T(1)}$ is an algebra over $\bS^{\tsadi_1}_{2}$.  In fact, this difficulty is essential: we have by \Cref{thm:mainprecise} that $|BC_2|$ is the identity on $\bS^{\tsadi_1}_2$, but it is shown in \cite{CY} that $|BC_2|_{\bS_{T(1)}}$ is multiplication by the element
\[
1+\epsilon \in \pi_0 \bS_{T(1)} \cong \Z_2[\epsilon]/(\epsilon^2, 2\epsilon),
\]
so there cannot be a ring map $\bS^{\tsadi_1}_2 \to U_{T(1)}^{\tsadi}(\bS_{T(1)})$.    
\end{rmk}

\subsection{Higher cyclotomic extensions of $\bS^{\tsadi_1}_p$}\label{sub:cycl}


Let $(\C,\one)\in \CAlg_{\tsadi_n}(\PrL)$ be a $p$-typical $\infty$-semiadditively symmetric monoidal $\infty$-category of height $n$.  Then Carmeli--Schlank--Yanovski show that for any $r\geq 1$, there is an idempotent 
\[
\varepsilon_r \in \pi_0 \one[B^nC_{p^r}] := \pi_0 \Map(\one, \one[B^nC_{p^r}])
\]
such that the map $\one[B^nC_{p^r}] \to \one [B^nC_{p^{r-1}}]$ induced by the surjection $C_{p^r} \twoheadrightarrow C_{p^{r-1}}$ can be identified with the map inverting $\varepsilon_r$ \cite[Proposition 4.5]{CSYCyc}.  

\begin{dfn}\label{dfn:rootsofunity}
In the above situation, define for any $R\in \CAlg(\C)$
\[
R[\omega^{(n)}_{p^r}] := R[B^nC_{p^r}][(1-\varepsilon_r)^{-1}]
\]
to be the \emph{height $n$ $p^r$-th cyclotomic extension} of $R$.  
\end{dfn}


\begin{exm}
Consider the case $\C = \Mod_{\Q}$ at height $0$.  Then we have
\[
1-\varepsilon_1 = \frac{1}{p}(1+\gamma +\cdots +\gamma^{p-1}) \in \Q[\gamma]/(\gamma^p-1) \cong \Q[C_p]
\]
and so the height $0$ $p$-th cyclotomic extension
\[
\Q[\zeta_p^{(0)}] = \Q[C_p][(1-\varepsilon_1)^{-1}] = \Q[\zeta_p]
\]
is the usual $p$-th cyclotomic extension.  
\end{exm}

A basic fact in the classical situation is that the $p^r$-th cyclotomic extensions of a field of characteristic zero are \emph{Galois}.  Carmeli--Schlank--Yanovski show that, in fact, an analogous property holds in the $T(n)$-local (or $K(n)$-local) setting:

\begin{thm}[{\cite[Proposition 5.2]{CSYCyc}}]
For $R\in \CAlg(\Sp_{T(n)})$ and $r\geq 1$, the map
\[
R \to R[\omega^{(n)}_{p^r}]
\]
is a $(\Z/p^r)^{\times}$-Galois extension in $\CAlg(\Sp_{T(n)})$, in the sense of Rognes \cite{RognesGal}.
\end{thm}

One can wonder whether the height $n$ cyclotomic extensions are Galois for a general $\C \in  \CAlg_{\tsadi_n}(\PrL)$ in place of $\Sp_{T(n)}$.  As was explained to the author by Carmeli, Schlank, and Yanovski (and alluded to in \cite[\S 4]{CSYCyc}), the existence of $\bS^{\tsadi_1}_p$ shows that this is not the case:

\begin{prop}\label{prop:notgal}
The map of commutative algebras
\[
\bS^{\tsadi_1}_p \to \bS^{\tsadi_1}_p[\omega^{(1)}_p]
\]
is not a $(\Z/p)^{\times}$-Galois extension in $\CAlg(\tsadi_1)$.\footnote{The analogous statement with the $p^r$-th cyclotomic extensions is slightly more combinatorially involved, but can easily be checked along the same lines.}
\end{prop}

To prove this statement, we first note the following consequence of the Segal conjecture:

\begin{lem}\label{lem:seg}
For a finite $p$-group $G$, the group
\[
\pi_0(\bS^{\tsadi_1}_p [BG]) := \pi_0 \Map(\one_{\tsadi_1}, \bS^{\tsadi_1}_p ) \cong \pi_0 U_{\tsadi_1}(\bS^{\tsadi_1}_p [BG])
\]
is a free $\Z_p$-module of rank equal to the number of conjugacy classes of subgroups of $G$.
\end{lem}
\begin{proof}
Since $\tsadi_1$ is $\infty$-semiadditive and $U_{\tsadi_1}$ preserves limits, we have
\[
 U_{\tsadi_1}(\bS^{\tsadi_1}_p [BG]) \simeq  U_{\tsadi_1}((\bS^{\tsadi_1}_p)^{BG}) \simeq U_{\tsadi_1}(\bS^{\tsadi_1}_p)^{BG} \simeq \bS_p^{BG}.
\]
But by Carlsson's theorem \cite{Carlsson}, this has $\pi_0$ isomorphic to the $p$-completed Burnside ring of finite $G$-sets, so the conclusion follows.  
\end{proof}

\begin{proof}[Proof of \Cref{prop:notgal}]
By \Cref{dfn:rootsofunity}, we have a splitting
\begin{equation}\label{eqn:1}
\bS^{\tsadi_1}_p[BC_p] \simeq \bS^{\tsadi_1}_p \oplus \bS^{\tsadi_1}_p[\omega^{(1)}_p].
\end{equation}
Now suppose that the extension $\bS^{\tsadi_1}_p \to \bS^{\tsadi_1}_p[\omega^{(1)}_p]$ were $(\Z/p)^{\times}$-Galois, so that
\[
\bS^{\tsadi_1}_p[\omega^{(1)}_p]\otimes_{\bS^{\tsadi_1}_p}\bS^{\tsadi_1}_p[\omega^{(1)}_p] \simeq \bS^{\tsadi_1}_p[\omega^{(1)}_p]^{(\Z/p)^{\times}}.
\]
Then, taking the tensor square of (\ref{eqn:1}) over $\bS^{\tsadi_1}_p$, we would have
\begin{align*}
\bS^{\tsadi_1}_p[BC_p\times BC_p] &\simeq \bS^{\tsadi_1}_p \oplus \bS^{\tsadi_1}_p[\omega^{(1)}_p]^{\oplus 2} \oplus \bS^{\tsadi_1}_p[\omega^{(1)}_p]^{\otimes 2} \\
&\simeq \bS^{\tsadi_1}_p \oplus \bS^{\tsadi_1}_p[\omega^{(1)}_p]^{\oplus 2} \oplus \bS^{\tsadi_1}_p[\omega^{(1)}_p]^{(\Z/p)^{\times}}\\
&\simeq \bS^{\tsadi_1}_p\oplus \bS^{\tsadi_1}_p[\omega^{(1)}_p]^{\oplus p+1}.
\end{align*}
But by comparing $\pi_0$, we see that this is impossible: namely, using \Cref{lem:seg} and (\ref{eqn:1}), we have that $\pi_0 \bS^{\tsadi_1}_p[\omega^{(1)}_p] = \Z_p$, so the right-hand side has $\pi_0$ isomorphic to $\Z_p^{\oplus p+2}$; on the other hand, by \Cref{lem:seg}, the left-hand side has $\pi_0$ isomorphic to $\Z_p^{\oplus p+3}$.
\end{proof}

\bibliographystyle{alpha}
\bibliography{Bibliography}

\end{document}